\documentclass{amsart}
\usepackage{amsfonts}

\newtheorem{theorem}{Theorem}
\theoremstyle{plain}

\newtheorem{corollary}{Corollary}

\newtheorem{definition}{Definition}

\numberwithin{equation}{section}
\input{tcilatex}

\begin{document}
\title[inequalities for $s$-logarithmically convex functions]{On some
integral inequalities for $s$-logarithmically convex functions and their
applications}
\author{Ahmet Ocak AKDEM\.{I}R$^{1}$}
\address{$^{1}$Department of Mathematics, Faculty of Science and Arts,
University of A\u{g}r\i\ \.{I}brahim \c{C}e\c{c}en, 04000, A\u{g}r\i , Turkey%
}
\email{$^{1}$ahmetakdemir@agri.edu.tr}
\author{Mevl\"{u}t TUN\c{C}$^{2}$}
\address{$^{2}$Department of Mathematics, Faculty of Science and Arts,
University of Kilis 7 Aral\i k, 79000, Kilis, Turkey}
\email{$^{2}$mevluttunc@kilis.edu.tr}
\date{October, 2012}
\thanks{This paper is in final form and no version of it will be submitted
for publication elsewhere.\\
$^{2}$Corresponding Author}

\begin{abstract}
In this paper, we describe $s$-logarithmically convex functions in the first
and second sense which are connected with the ordinary logatihmic convex and 
$s$-convex in the first and second sense. Afterwards, some new inequalities
related to above new definitions are given.
\end{abstract}

\maketitle

\section{Intoruction}

In this section we will present definitions and some results used in this
paper. In what follows, $I$ will be used to denote an interval of real
numbers.

\begin{definition}
Let $I$ be an interval in $%
%TCIMACRO{\U{211d} }%
%BeginExpansion
\mathbb{R}
%EndExpansion
.$ Then $f:I\rightarrow 
%TCIMACRO{\U{211d} }%
%BeginExpansion
\mathbb{R}
%EndExpansion
,$ $\emptyset \neq I\subseteq 
%TCIMACRO{\U{211d} }%
%BeginExpansion
\mathbb{R}
%EndExpansion
$\ is said to be convex if 
\begin{equation}
f\left( tx+\left( 1-t\right) y\right) \leq tf\left( x\right) +\left(
1-t\right) f\left( y\right) .  \label{1}
\end{equation}%
for all $x,y\in I$ and $t\in \left[ 0,1\right] .$
\end{definition}

The following concept was introduced by Ozlicz in the paper \cite{orl} and
was used in the theory of Ozlicz spaces:

\begin{definition}
\cite{orl} Let $0<s\leq 1$. A function $f:%
%TCIMACRO{\U{211d} }%
%BeginExpansion
\mathbb{R}
%EndExpansion
_{+}\rightarrow 
%TCIMACRO{\U{211d} }%
%BeginExpansion
\mathbb{R}
%EndExpansion
$\ where $%
%TCIMACRO{\U{211d} }%
%BeginExpansion
\mathbb{R}
%EndExpansion
_{+}:=\left[ 0,\infty \right) ,$\ is said to be $s$-convex in the first
sense if%
\begin{equation}
f\left( \alpha u+\beta v\right) \leq \alpha ^{s}f\left( u\right) +\beta
^{s}f\left( v\right) 
\end{equation}%
for all $u,v\in 
%TCIMACRO{\U{211d} }%
%BeginExpansion
\mathbb{R}
%EndExpansion
_{+}$\ and $\alpha ,\beta \geq 0$ with $\alpha ^{s}+\beta ^{s}=1.$\ It is
denoted this by $f\in K_{s}^{1}.$
\end{definition}

In the paper \cite{hud}, H. Hudzik and L. Maligranda considered, among
others, the following class of functions:

\begin{definition}
\cite{hud}\ A function $f:%
%TCIMACRO{\U{211d} }%
%BeginExpansion
\mathbb{R}
%EndExpansion
_{+}\rightarrow 
%TCIMACRO{\U{211d} }%
%BeginExpansion
\mathbb{R}
%EndExpansion
$\ is said to be $s$-convex in the second sense if \ \ \ \ \ \ \ \ \ \ \ \ 
\begin{equation}
f\left( \alpha u+\beta v\right) \leq \alpha ^{s}f\left( u\right) +\beta
^{s}f\left( v\right)   \label{2}
\end{equation}%
for all $u,v\geq 0$\ and $\alpha ,\beta \geq 0$ with $\alpha +\beta =1$ and $%
s$ fixed in $\left( 0,1\right] $. They denoted this by $f\in K_{s}^{2}.$
\end{definition}

It can be easily checked for $s=1$, $s$-convexity reduces to the ordinary
convexity of functions defined on $\left[ 0,\infty \right) $.\ \ \ \ \ \ \ \
\ \ \ 

\begin{definition}
\cite{2}A function $f:I\rightarrow \left[ 0,\infty \right) $ is said to be
log-convex or multiplicatively convex if $\log f$ is convex, or
equivalently, if for all $x,y\in I$ and $t\in \left[ 0,1\right] ,$ one has
the inequality:%
\begin{equation}
f\left( tx+\left( 1-t\right) y\right) \leq \left[ f\left( x\right) \right]
^{t}\left[ f\left( y\right) \right] ^{1-t}.  \label{1.2}
\end{equation}
\end{definition}

We note that if $f$ and $g$ are convex functions and $g$ is monotonic
nondecreasing, then $gof$ is convex. Moreover, since $f=exp(\log f)$, it
follows that a log-convex function is convex, but the converse is not true [%
\ref{2}, p. 7]. This fact is obvious from (\ref{1.2}) as by the
arithmetic-geometric mean inequality, we have%
\begin{equation}
\left[ f\left( x\right) \right] ^{t}\left[ f\left( y\right) \right]
^{1-t}\leq tf\left( x\right) +\left( 1-t\right) f\left( y\right) 
\label{1.3}
\end{equation}%
for all $x,y\in I$ and $t\in \left[ 0,1\right] $.

If the above inequality (\ref{1.2}) is reversed, then $f$ is called
logarithmically concave, or simply $log$-concave. Apparently, it would seem
that $log$-concave ($log$-convex) functions would be unremarkable because
they are simply related to concave (convex) functions. But they have some
surprising properties. It is well known that the product of $log$-concave ($%
log$-convex) functions is also $log$-concave ($log$-convex). Moreover, the
sum of $log$-convex functions is also $log$-convex, and a convergent
sequence of $log$-convex ($log$-concave) functions has a $log$-convex ($log$%
-concave) limit function provided the limit is positive. However, the sum of 
$log$-concave functions is not necessarily $log$-concave. Due to their
interesting properties, the $log$-convex ($log$-concave) functions
frequently appear in many problems of classical analysis and probability
theory.

The next inequality is well known in the literature as the Hermite-Hadamard
inequality for convex functions%
\begin{equation}
f\left( \frac{a+b}{2}\right) \leq \frac{1}{b-a}\int_{a}^{b}f\left( x\right)
dx\leq \frac{f\left( a\right) +f\left( b\right) }{2}  \label{1.4}
\end{equation}%
where $f:I\rightarrow 
%TCIMACRO{\U{211d} }%
%BeginExpansion
\mathbb{R}
%EndExpansion
$ is a convex function on the interval $I$ of real numbers $a,b\in I$ with $%
a<b.$

For some recent results related to this classic result, see the papers \cite%
{hud}-\cite{10} and the books \cite{2}-\cite{dr2} where further references
are given.

In \cite{8}, S.S. Dragomir and B. Mond proved that the following
inequalities of Hermite-Hadamard type hold for $log$-convex functions:

\begin{theorem}
Let $f:I\rightarrow \left[ 0,\infty \right) $ be a log-convex mapping on $I$
and $a,b\in I$ with $a<b$. Then one has the inequality:%
\begin{equation}
f\left( A\left( a,b\right) \right) \leq \frac{1}{b-a}\int_{a}^{b}G\left(
f\left( x\right) ,f\left( a+b-x\right) \right) dx\leq G\left( f\left(
a\right) ,f\left( b\right) \right) .  \label{dr1}
\end{equation}
\end{theorem}

\begin{theorem}
\bigskip Let $f:I\rightarrow \left( 0,\infty \right) $ be a log-convex
mapping on $I$ and $a,b\in I$ with $a<b$. Then one has the inequality:%
\begin{eqnarray}
f\left( \frac{a+b}{2}\right) &\leq &\exp \left[ \frac{1}{b-a}\int_{a}^{b}\ln %
\left[ f\left( x\right) \right] dx\right]  \label{dr2} \\
&\leq &\frac{1}{b-a}\int_{a}^{b}G\left( f\left( x\right) ,f\left(
a+b-x\right) \right) dx\leq \frac{1}{b-a}\int_{a}^{b}f\left( x\right) dx 
\notag \\
&\leq &L\left( f\left( a\right) ,f\left( b\right) \right) \leq \frac{f\left(
a\right) +f\left( b\right) }{2}  \notag
\end{eqnarray}%
where $G(p,q):=\sqrt{pq}$ is the geometric mean and $L(p,q):=\frac{p\text{-}q%
}{\ln p-\ln q}$ $(p\neq q)$ is the logarithmic mean of the strictly positive
real numbers $p,q,$ i.e., 
\begin{equation*}
L\left( p,q\right) =\frac{p-q}{\ln p-\ln q}\text{ }if\text{ }p\neq q\text{
and }L(p,p)=p.
\end{equation*}
\end{theorem}

In \cite{10}, B.G. Pachpatte proved that the inequalities hold for two $log$%
-convex functions:%
\begin{eqnarray}
\frac{4}{b-a}\int_{a}^{b}f\left( x\right) g\left( x\right) dx &\leq &\left[
f\left( a\right) +f\left( b\right) \right] L\left( f\left( a\right) ,f\left(
b\right) \right)  \label{1.6} \\
&&+\left[ g\left( a\right) +g\left( b\right) \right] L\left( g\left(
a\right) ,g\left( b\right) \right)  \notag
\end{eqnarray}

Recently, In \cite{bai}, the concept of geometrically and $m$- and $\left(
\alpha ,m\right) $-logarithmically convex functions was introduced as
follows.

\begin{definition}
\label{d}A function $f:[0,b]\rightarrow (0,\infty )$ is said to be $m$%
-logarithmically convex if the inequality%
\begin{equation}
f\left( tx+m\left( 1-t\right) y\right) \leq \left[ f\left( x\right) \right]
^{t}\left[ f\left( y\right) \right] ^{m\left( 1-t\right) }  \label{d1}
\end{equation}%
holds for all $x,y\in \lbrack 0,b]$, $m\in (0,1]$, and $t\in \lbrack 0,1]$.
\end{definition}

Obviously, if putting $m=1$ in Definition \ref{d}, then $f$ is just the
ordinary logarithmically convex function on $\left[ 0,b\right] $.

\begin{definition}
\label{dx}A function $f:[0,b]\rightarrow (0,\infty )$ is said to be $\left(
\alpha ,m\right) $-logarithmically convex if%
\begin{equation}
f\left( tx+m\left( 1-t\right) y\right) \leq \left[ f\left( x\right) \right]
^{t^{\alpha }}\left[ f\left( y\right) \right] ^{m\left( 1-t^{\alpha }\right)
}  \label{d2}
\end{equation}%
holds for all $x,y\in \lbrack 0,b]$, $\left( \alpha ,m\right) \in \left( 0,1%
\right] \times \left( 0,1\right] ,$ and $t\in \lbrack 0,1]$.
\end{definition}

Clearly, when taking $\alpha =1$ in Definition \ref{dx}, then $f$ becomes
the standard $m$-logarithmically convex function on $\left[ 0,b\right] $.

The main purpose of this paper is to introduce the concepts of $s$%
-logarithmically convex in the first and second sense and to establish some
inequalities of Hadamard type for $s$-logarithmically convex functions.

\section{Definitions of $s$-logarithmically convex functions}

Now it is time to introduce two new classes of functions which will be
called $s$-logarithmically convex in the first and second sense.

\begin{definition}
\label{mm}A function $f:I\subset 
%TCIMACRO{\U{211d} }%
%BeginExpansion
\mathbb{R}
%EndExpansion
_{0}\rightarrow 
%TCIMACRO{\U{211d} }%
%BeginExpansion
\mathbb{R}
%EndExpansion
_{+}$\ is said to be $s$-logarithmically convex in the first sense\ if \ \ \
\ \ \ \ \ \ \ \ \ 
\begin{equation}
f\left( \alpha x+\beta y\right) \leq \left[ f\left( x\right) \right]
^{\alpha ^{s}}\left[ f\left( y\right) \right] ^{\beta ^{s}}  \label{m1}
\end{equation}%
for some $s\in \left( 0,1\right] $, where $x,y\in I$\ and $\alpha ^{s}+\beta
^{s}=1.$
\end{definition}

\begin{definition}
\label{mmm}A function $f:I\subset 
%TCIMACRO{\U{211d} }%
%BeginExpansion
\mathbb{R}
%EndExpansion
_{0}\rightarrow 
%TCIMACRO{\U{211d} }%
%BeginExpansion
\mathbb{R}
%EndExpansion
_{+}$\ is said to be $s$-logarithmically convex in the second sense\ if \ \
\ \ \ \ \ \ \ \ \ \ 
\begin{equation}
f\left( tx+\left( 1-t\right) y\right) \leq \left[ f\left( x\right) \right]
^{t^{s}}\left[ f\left( y\right) \right] ^{\left( 1-t\right) ^{s}}  \label{m2}
\end{equation}%
for some $s\in \left( 0,1\right] $, where $x,y\in I$\ and $t\in \left[ 0,1%
\right] $.
\end{definition}

Clearly, when taking $s=1$ in Definition \ref{m1} or Definition \ref{m2},
then $f$ becomes the standard logarithmically convex function on $I$. If the
above inequalities (\ref{m1}) and (\ref{m2}) are reversed, then $f$ is
called $s$-logarithmically concave in the first and second sense,
respctively.

\section{On some inequalities for $s$-logarithmically convex function in the
first sense}

\begin{theorem}
\bigskip Let $f:I\subset 
%TCIMACRO{\U{211d} }%
%BeginExpansion
\mathbb{R}
%EndExpansion
_{0}\rightarrow 
%TCIMACRO{\U{211d} }%
%BeginExpansion
\mathbb{R}
%EndExpansion
_{+}$\ be a $s$-logarithmically convex mapping in the first sense and
monotonic nondecreasing on $I$ with $s\in \left( 0,1\right] .$\ If $a,b\in I$
with $a<b,$ then the following inequality holds:\ \ \ \ \ \ \ \ \ \ \ \ \ 
\begin{equation}
f\left( \frac{a+b}{2^{\frac{1}{s}}}\right) \leq \frac{1}{b-a}%
\int_{a}^{b}G\left( f\left( x\right) ,f\left( a+b-x\right) \right) dx
\label{111}
\end{equation}%
where $G(p,q):=\sqrt{pq}$ is geometric mean of the strictly positive real
numbers $p,q$.
\end{theorem}

\begin{proof}
\bigskip If we choose in the definition of $s$-logarithmically convex
mapping in the first sense $\alpha =\frac{1}{2^{\frac{1}{s}}},$ $\beta =%
\frac{1}{2^{\frac{1}{s}}},$ we have that $\alpha ^{s}+\beta ^{s}=1$ and then
for all $x,y\in I$%
\begin{equation*}
f\left( \frac{x+y}{2^{\frac{1}{s}}}\right) \leq G\left( f\left( x\right)
,f\left( y\right) \right) .
\end{equation*}%
If we choose $x=ta+\left( 1-t\right) b,$ $y=tb+\left( 1-t\right) a,$ $t\in %
\left[ 0,1\right] ,$ we obtain%
\begin{equation*}
f\left( \frac{a+b}{2^{\frac{1}{s}}}\right) \leq G\left( f\left( ta+\left(
1-t\right) b\right) ,f\left( tb+\left( 1-t\right) a\right) \right) .
\end{equation*}%
By integrating over $t$ on $\left[ 0,1\right] $ in the above inequality. \
The proof is completed.
\end{proof}

\begin{theorem}
Let $f:I\subset 
%TCIMACRO{\U{211d} }%
%BeginExpansion
\mathbb{R}
%EndExpansion
_{0}\rightarrow 
%TCIMACRO{\U{211d} }%
%BeginExpansion
\mathbb{R}
%EndExpansion
_{+}$\ be a $s$-logarithmically convex mapping in the first sense and
monotonic nondecreasing on $I$ with $s\in \left( 0,1\right] .$\ If $a,b\in I$
with $a<b,$ and $\frac{1}{p}+\frac{1}{q}=1$ with $p<0$ or $q<0$, then the
following inequality holds:%
\begin{eqnarray}
&&  \label{11} \\
&&\left( \int_{0}^{1}\left[ f\left( ta+\left( 1-t^{s}\right) ^{\frac{1}{s}%
}b\right) \right] ^{p}dt\right) ^{\frac{1}{p}}\left( \int_{0}^{1}\left[
f\left( ta+\left( 1-t^{s}\right) ^{\frac{1}{s}}b\right) \left(
1-t^{s}\right) ^{\frac{1}{s}-1}t^{s-1}\right] ^{q}dt\right) ^{\frac{1}{q}} 
\notag \\
&\leq &\left[ f\left( a\right) \right] \left[ f\left( b\right) \right]  
\notag
\end{eqnarray}
\end{theorem}

\begin{proof}
\bigskip If we choose in the Definition \ref{mm} $\alpha =t,$ $\beta =\left(
1-t^{s}\right) ^{\frac{1}{s}},$ $t\in \left[ 0,1\right] ,$ we have that $%
\alpha ^{s}+\beta ^{s}=1$ for all $t\in \left[ 0,1\right] $ and%
\begin{equation}
f\left( ta+\left( 1-t^{s}\right) ^{\frac{1}{s}}b\right) \leq \left[ f\left(
a\right) \right] ^{t^{s}}\left[ f\left( b\right) \right] ^{\left(
1-t^{s}\right) }  \label{12}
\end{equation}%
for all $t\in \left[ 0,1\right] ,$ and similarly%
\begin{equation}
f\left( \left( 1-t^{s}\right) ^{\frac{1}{s}}a+tb\right) \leq \left[ f\left(
a\right) \right] ^{\left( 1-t^{s}\right) }\left[ f\left( b\right) \right]
^{t^{s}}  \label{13}
\end{equation}%
for all $t\in \left[ 0,1\right] .$

If we multiply the above two inequalities, we have that%
\begin{equation*}
f\left( ta+\left( 1-t^{s}\right) ^{\frac{1}{s}}b\right) f\left( \left(
1-t^{s}\right) ^{\frac{1}{s}}a+tb\right) \leq \left[ f\left( a\right) \right]
\left[ f\left( b\right) \right]
\end{equation*}%
for all $t\in \left[ 0,1\right] .$ If we integrate this inequality over $t$
on $\left[ 0,1\right] $, we get that%
\begin{equation*}
\int_{0}^{1}f\left( ta+\left( 1-t^{s}\right) ^{\frac{1}{s}}b\right) f\left(
\left( 1-t^{s}\right) ^{\frac{1}{s}}a+tb\right) dt\leq \left[ f\left(
a\right) \right] \left[ f\left( b\right) \right] .
\end{equation*}%
Using H\"{o}lder inequality, we have%
\begin{eqnarray}
&&\left( \int_{0}^{1}\left[ f\left( ta+\left( 1-t^{s}\right) ^{\frac{1}{s}%
}b\right) \right] ^{p}dt\right) ^{\frac{1}{p}}\left( \int_{0}^{1}\left[
f\left( \left( 1-t^{s}\right) ^{\frac{1}{s}}a+tb\right) \right]
^{q}dt\right) ^{\frac{1}{q}}  \notag \\
&\leq &\int_{0}^{1}f\left( ta+\left( 1-t^{s}\right) ^{\frac{1}{s}}b\right)
f\left( \left( 1-t^{s}\right) ^{\frac{1}{s}}a+tb\right) dt  \label{k}
\end{eqnarray}%
Let us denote $u=\left( 1-t^{s}\right) ^{\frac{1}{s}},$ $t\in \left[ 0,1%
\right] .$ Then $t=\left( 1-u^{s}\right) ^{\frac{1}{s}}$ and $dt=-\left(
1-u^{s}\right) ^{\frac{1}{s}-1}u^{s-1}du,$ $u\in \left( 0,1\right] $ and
then we have the change of variable%
\begin{eqnarray*}
&&\int_{0}^{1}f\left( \left( 1-t^{s}\right) ^{\frac{1}{s}}a+tb\right) dt \\
&=&-\int_{1}^{0}f\left( ua+\left( 1-u^{s}\right) ^{\frac{1}{s}}b\right)
\left( 1-u^{s}\right) ^{\frac{1}{s}-1}u^{s-1}du \\
&=&\int_{0}^{1}f\left( ta+\left( 1-t^{s}\right) ^{\frac{1}{s}}b\right)
\left( 1-t^{s}\right) ^{\frac{1}{s}-1}t^{s-1}dt.
\end{eqnarray*}%
Using the inequality (\ref{k}) , we get that%
\begin{eqnarray*}
&&\left( \int_{0}^{1}\left[ f\left( ta+\left( 1-t^{s}\right) ^{\frac{1}{s}%
}b\right) \right] ^{p}dt\right) ^{\frac{1}{p}}\left( \int_{0}^{1}\left[
f\left( ta+\left( 1-t^{s}\right) ^{\frac{1}{s}}b\right) \left(
1-t^{s}\right) ^{\frac{1}{s}-1}t^{s-1}\right] ^{q}dt\right) ^{\frac{1}{q}} \\
&\leq &\left[ f\left( a\right) \right] \left[ f\left( b\right) \right]
\end{eqnarray*}%
and the proof is completed.
\end{proof}

\section{On some inequalities for $s$-logarithmically convex function in the
second sense}

\begin{theorem}
\label{t1}Let $f:I\subset 
%TCIMACRO{\U{211d} }%
%BeginExpansion
\mathbb{R}
%EndExpansion
_{0}\rightarrow 
%TCIMACRO{\U{211d} }%
%BeginExpansion
\mathbb{R}
%EndExpansion
_{+}$\ be a $s$-logarithmically convex mapping in the second sense on $I$
with $s\in \left( 0,1\right] .$\ If $a,b\in I$ with $a<b,$ then the
following inequality holds:%
\begin{equation}
\frac{1}{b-a}\int_{a}^{b}f\left( x\right) dx\leq \int_{0}^{1}\left[ f\left(
a\right) \right] ^{t^{s}}\left[ f\left( b\right) \right] ^{\left( 1-t\right)
^{s}}dt.  \label{222}
\end{equation}%
for all $t\in \left[ 0,1\right] .$\ \ \ 
\end{theorem}

\begin{proof}
Since $f$ is $s$-logarithmically convex mapping in the second sense, we
have, for all $t\in \left[ 0,1\right] $%
\begin{equation*}
f\left( ta+\left( 1-t\right) b\right) \leq \left[ f\left( a\right) \right]
^{t^{s}}\left[ f\left( b\right) \right] ^{\left( 1-t\right) ^{s}}
\end{equation*}%
Integrating this inequality over $t$ on $\left[ 0,1\right] $ , we get%
\begin{equation*}
\int_{0}^{1}f\left( ta+\left( 1-t\right) b\right) dt\leq \int_{0}^{1}\left[
f\left( a\right) \right] ^{t^{s}}\left[ f\left( b\right) \right] ^{\left(
1-t\right) ^{s}}dt.
\end{equation*}%
As the change of variable $x=ta+\left( 1-t\right) b$ gives us that%
\begin{equation*}
\int_{0}^{1}f\left( ta+\left( 1-t\right) b\right) dt=\frac{1}{b-a}%
\int_{a}^{b}f\left( x\right) dx,
\end{equation*}%
The proof is completed.
\end{proof}

\begin{theorem}
\bigskip \bigskip Under the assumptions of Theorem \ref{t1}, the following
inequality holds:%
\begin{equation}
\frac{1}{b-a}\int_{a}^{b}f\left( x\right) dx\leq K\left( s,k\left( \mu
\right) \right)  \label{kk}
\end{equation}%
where%
\begin{equation*}
\mu \left( u,v\right) =\left[ f\left( a\right) \right] ^{u}\left[ f\left(
b\right) \right] ^{-v},u,v>0,
\end{equation*}%
\begin{equation*}
k\left( \mu \right) =\left\{ 
\begin{array}{c}
1,\text{ \ \ \ \ \ \ }\mu =1, \\ 
\frac{\mu -1}{\ln \mu },\text{ \ \ \ \ \ \ }\mu \neq 1,%
\end{array}%
\right.
\end{equation*}%
and%
\begin{equation*}
K\left( s,k\left( \mu \right) \right) =\left[ f\left( b\right) \right]
^{s}k\left( \mu \left( s,s\right) \right) ,\text{ \ \ \ \ }f\left( a\right)
,f\left( b\right) \leq 1.
\end{equation*}%
\ \ \ 
\end{theorem}

\begin{proof}
Since $f$ is $s$-logarithmically convex mapping in the second sense, we
have, for all $t\in \left[ 0,1\right] $%
\begin{equation*}
f\left( ta+\left( 1-t\right) b\right) \leq \left[ f\left( a\right) \right]
^{t^{s}}\left[ f\left( b\right) \right] ^{\left( 1-t\right) ^{s}}
\end{equation*}%
Integrating this inequality over $t$ on $\left[ 0,1\right] $ , we get%
\begin{equation*}
\int_{0}^{1}f\left( ta+\left( 1-t\right) b\right) dt\leq \int_{0}^{1}\left[
f\left( a\right) \right] ^{t^{s}}\left[ f\left( b\right) \right] ^{\left(
1-t\right) ^{s}}dt.
\end{equation*}%
If $0<\rho \leq 1,$ $0<t,s\leq 1,$ then%
\begin{equation}
\rho ^{t^{s}}\leq \rho ^{ts}.  \label{b}
\end{equation}%
When $f\left( a\right) ,$ $f\left( b\right) \leq 1,$ by (\ref{b}), we get
that%
\begin{eqnarray}
\int_{0}^{1}\left[ f\left( a\right) \right] ^{t^{s}}\left[ f\left( b\right) %
\right] ^{\left( 1-t\right) ^{s}}dt &\leq &\int_{0}^{1}\left[ f\left(
a\right) \right] ^{st}\left[ f\left( b\right) \right] ^{s\left( 1-t\right)
}dt  \label{k1} \\
&=&\left[ f\left( b\right) \right] ^{s}\int_{0}^{1}\left[ f\left( a\right) %
\right] ^{st}\left[ f\left( b\right) \right] ^{-st}dt  \notag \\
&=&\left[ f\left( b\right) \right] ^{s}k\left( \mu \left( s,s\right) \right)
.  \notag
\end{eqnarray}%
As the change of variable $x=ta+\left( 1-t\right) b$ gives us that%
\begin{equation}
\int_{0}^{1}f\left( ta+\left( 1-t\right) b\right) dt=\frac{1}{b-a}%
\int_{a}^{b}f\left( x\right) dx,  \label{k5}
\end{equation}%
from (\ref{k1}) to (\ref{k5}), (\ref{kk}) holds.
\end{proof}

\begin{theorem}
\label{f} \bigskip \bigskip Let $f,g:I\subset 
%TCIMACRO{\U{211d} }%
%BeginExpansion
\mathbb{R}
%EndExpansion
_{0}\rightarrow 
%TCIMACRO{\U{211d} }%
%BeginExpansion
\mathbb{R}
%EndExpansion
_{+}$\ be a $s$-logarithmically convex mappings in the second sense on $I$
with $s\in \left( 0,1\right] .$\ If $a,b\in I$ with $a<b,$ then the
following inequality holds:%
\begin{equation}
\frac{1}{b-a}\int_{a}^{b}f\left( x\right) g\left( x\right) dx\leq K\left(
s,k\left( \eta \right) \right)  \label{ff}
\end{equation}%
where%
\begin{equation*}
\eta \left( u,v\right) =\left[ f\left( a\right) g\left( a\right) \right] ^{u}%
\left[ f\left( b\right) g\left( b\right) \right] ^{-v},u,v>0,
\end{equation*}%
\begin{equation*}
k\left( \eta \right) =\left\{ 
\begin{array}{c}
1,\text{ \ \ \ \ \ \ }\eta =1, \\ 
\frac{\eta -1}{\ln \eta },\text{ \ \ \ \ \ \ }\eta \neq 1,%
\end{array}%
\right.
\end{equation*}%
and%
\begin{equation*}
K\left( s,k\left( \eta \right) \right) =\left[ f\left( b\right) g\left(
b\right) \right] ^{s}k\left( \eta \left( s,s\right) \right) ,\text{ \ \ \ \ }%
f\left( a\right) g\left( a\right) ,\text{ }f\left( b\right) g\left( b\right)
\leq 1.
\end{equation*}%
\ \ \ 
\end{theorem}

\begin{proof}
Since $f,g$ are $s$-logarithmically convex mappings in the second sense, we
have, for all $t\in \left[ 0,1\right] $%
\begin{equation*}
f\left( ta+\left( 1-t\right) b\right) \leq \left[ f\left( a\right) \right]
^{t^{s}}\left[ f\left( b\right) \right] ^{\left( 1-t\right) ^{s}}\text{ and }%
g\left( ta+\left( 1-t\right) b\right) \leq \left[ g\left( a\right) \right]
^{t^{s}}\left[ g\left( b\right) \right] ^{\left( 1-t\right) ^{s}}
\end{equation*}%
from which it follows that%
\begin{eqnarray}
\int_{a}^{b}f\left( x\right) g\left( x\right) dx &=&\left( b-a\right)
\int_{0}^{1}f\left( ta+\left( 1-t\right) b\right) g\left( ta+\left(
1-t\right) b\right) dt  \label{f1} \\
&\leq &\left( b-a\right) \int_{0}^{1}\left[ f\left( a\right) g\left(
a\right) \right] ^{t^{s}}\left[ f\left( b\right) g\left( b\right) \right]
^{\left( 1-t\right) ^{s}}dt  \notag
\end{eqnarray}%
When $fg\left( a\right) ,$ $fg\left( b\right) \leq 1,$ by (\ref{b}), we get
that%
\begin{eqnarray}
\int_{0}^{1}\left[ f\left( a\right) g\left( a\right) \right] ^{t^{s}}\left[
f\left( b\right) g\left( b\right) \right] ^{\left( 1-t\right) ^{s}}dt &\leq
&\int_{0}^{1}\left[ f\left( a\right) g\left( a\right) \right] ^{st}\left[
f\left( b\right) g\left( b\right) \right] ^{s\left( 1-t\right) }dt
\label{f2} \\
&=&\left[ f\left( b\right) g\left( b\right) \right] ^{s}k\left( \eta \left(
s,s\right) \right) .  \notag
\end{eqnarray}%
from (\ref{f1}) to (\ref{f2}), (\ref{ff}) holds.
\end{proof}

\begin{corollary}
\bigskip Let $f,g:I\subset 
%TCIMACRO{\U{211d} }%
%BeginExpansion
\mathbb{R}
%EndExpansion
_{0}\rightarrow 
%TCIMACRO{\U{211d} }%
%BeginExpansion
\mathbb{R}
%EndExpansion
_{+}$\ be a log-convex mappings on $I$.\ If $a,b\in I$ with $a<b,$ then the
following inequality holds:%
\begin{equation*}
\frac{1}{b-a}\int_{a}^{b}f\left( x\right) g\left( x\right) dx\leq L\left(
f\left( a\right) g\left( a\right) ,\text{ }f\left( b\right) g\left( b\right)
\right)
\end{equation*}%
where $L(p,q):=\frac{p\text{-}q}{\ln p-\ln q}$ $(p\neq q)$ is the
logarithmic mean of the positive real numbers.
\end{corollary}

\begin{proof}
\bigskip We take $s=1$ in (\ref{ff}), we get the required result.
\end{proof}

\begin{theorem}
\bigskip \bigskip Let $f,g,a,b$\ be as in Theorem \ref{f} and $\alpha ,\beta
>0$ with $\alpha +\beta =1.$Then the following inequality holds:%
\begin{equation}
\frac{1}{b-a}\int_{a}^{b}f\left( x\right) g\left( x\right) dx\leq K\left(
s,\alpha ,\beta ;k\left( \omega \right) ,k\left( \ell \right) \right)
\label{oo}
\end{equation}%
where%
\begin{equation*}
\omega \left( u,v\right) =\left[ f\left( a\right) \right] ^{u}\left[ f\left(
b\right) \right] ^{-v},\text{ \ and \ }\ell \left( u,v\right) =\left[
g\left( a\right) \right] ^{u}\left[ g\left( b\right) \right] ^{-v},\text{\ }%
u,v>0,
\end{equation*}%
\begin{equation*}
k\left( \omega \right) =\left\{ 
\begin{array}{c}
1,\text{ \ \ \ \ \ \ }\omega =1, \\ 
\frac{\omega -1}{\ln \omega },\text{ \ \ \ \ \ \ }\omega \neq 1,%
\end{array}%
\right. ,\text{ }k\left( \ell \right) =\left\{ 
\begin{array}{c}
1,\text{ \ \ \ \ \ \ }\ell =1, \\ 
\frac{\ell -1}{\ln \ell },\text{ \ \ \ \ \ \ }\ell \neq 1,%
\end{array}%
\right.
\end{equation*}%
and%
\begin{eqnarray*}
&&K\left( s,\alpha ,\beta ;k\left( \omega \right) ,k\left( \ell \right)
\right) \\
&=&\alpha \left[ f\left( b\right) \right] ^{\frac{s}{\alpha }}k\left( \omega
\left( \frac{s}{\alpha },\frac{s}{\alpha }\right) \right) +\beta \left[
g\left( b\right) \right] ^{\frac{s}{\beta }}k\left( \ell \left( \frac{s}{%
\beta },\frac{s}{\beta }\right) \right) ,\text{ }f\left( a\right) ,g\left(
a\right) ,f\left( b\right) ,g\left( b\right) \leq 1.
\end{eqnarray*}%
\ \ \ 
\end{theorem}

\begin{proof}
Since $f,g$ are $s$-logarithmically convex mappings in the second sense, we
have, for all $t\in \left[ 0,1\right] $%
\begin{equation}
f\left( ta+\left( 1-t\right) b\right) \leq \left[ f\left( a\right) \right]
^{t^{s}}\left[ f\left( b\right) \right] ^{\left( 1-t\right) ^{s}}\text{ and }%
g\left( ta+\left( 1-t\right) b\right) \leq \left[ g\left( a\right) \right]
^{t^{s}}\left[ g\left( b\right) \right] ^{\left( 1-t\right) ^{s}}  \label{o0}
\end{equation}%
from which it follows that%
\begin{equation}
\int_{a}^{b}f\left( x\right) g\left( x\right) dx=\left( b-a\right)
\int_{0}^{1}f\left( ta+\left( 1-t\right) b\right) g\left( ta+\left(
1-t\right) b\right) dt  \label{o1}
\end{equation}%
Using the well known inequality $mn\leq \alpha m^{\frac{1}{\alpha }}+\beta
n^{\frac{1}{\beta }},$ (\ref{o0}), on the right side of (\ref{o1}), we have 
\begin{eqnarray*}
\frac{1}{b-a}\int_{a}^{b}f\left( x\right) g\left( x\right) dx &\leq
&\int_{0}^{1}\left\{ \alpha \left[ f\left( ta+\left( 1-t\right) b\right) %
\right] ^{\frac{1}{\alpha }}+\beta \left[ g\left( ta+\left( 1-t\right)
b\right) \right] ^{\frac{1}{\beta }}\right\} dt \\
&\leq &\int_{0}^{1}\left\{ \alpha \left[ \left[ f\left( a\right) \right]
^{t^{s}}\left[ f\left( b\right) \right] ^{\left( 1-t\right) ^{s}}\right] ^{%
\frac{1}{\alpha }}+\beta \left[ \left[ g\left( a\right) \right] ^{t^{s}}%
\left[ g\left( b\right) \right] ^{\left( 1-t\right) ^{s}}\right] ^{\frac{1}{%
\beta }}\right\} dt \\
&=&\alpha \int_{0}^{1}\left[ f\left( a\right) \right] ^{\frac{t^{s}}{\alpha }%
}\left[ f\left( b\right) \right] ^{\frac{\left( 1-t\right) ^{s}}{\alpha }%
}dt+\beta \int_{0}^{1}\left[ g\left( a\right) \right] ^{\frac{t^{s}}{\beta }}%
\left[ g\left( b\right) \right] ^{\frac{\left( 1-t\right) ^{s}}{\beta }}dt
\end{eqnarray*}%
When $f\left( a\right) ,$ $g\left( a\right) ,$ $f\left( b\right) ,$ $g\left(
b\right) \leq 1,$ by (\ref{b}), we get that%
\begin{eqnarray}
\alpha \int_{0}^{1}\left[ f\left( a\right) \right] ^{\frac{t^{s}}{\alpha }}%
\left[ f\left( b\right) \right] ^{\frac{\left( 1-t\right) ^{s}}{\alpha }}dt
&\leq &\alpha \int_{0}^{1}\left[ f\left( a\right) \right] ^{\frac{st}{\alpha 
}}\left[ f\left( b\right) \right] ^{\frac{s\left( 1-t\right) }{\alpha }}dt 
\notag \\
&=&\alpha \left[ f\left( b\right) \right] ^{\frac{s}{\alpha }}k\left( \omega
\left( \frac{s}{\alpha },\frac{s}{\alpha }\right) \right) .  \notag \\
\beta \int_{0}^{1}\left[ g\left( a\right) \right] ^{\frac{t^{s}}{\beta }}%
\left[ g\left( b\right) \right] ^{\frac{\left( 1-t\right) ^{s}}{\beta }}dt
&\leq &\beta \int_{0}^{1}\left[ g\left( a\right) \right] ^{\frac{st}{\beta }}%
\left[ g\left( b\right) \right] ^{\frac{s\left( 1-t\right) }{\beta }}dt 
\notag \\
&=&\beta \left[ g\left( b\right) \right] ^{\frac{s}{\beta }}k\left( \ell
\left( \frac{s}{\beta },\frac{s}{\beta }\right) \right) .  \label{o2}
\end{eqnarray}%
From (\ref{o1}) to (\ref{o2}), (\ref{oo}) holds.
\end{proof}

\begin{corollary}
\bigskip Let $f,g:I\subset 
%TCIMACRO{\U{211d} }%
%BeginExpansion
\mathbb{R}
%EndExpansion
_{0}\rightarrow 
%TCIMACRO{\U{211d} }%
%BeginExpansion
\mathbb{R}
%EndExpansion
_{+}$\ be a log-convex mappings on $I$.\ If $a,b\in I$ with $a<b,$ then the
following inequality holds:%
\begin{equation*}
\frac{1}{b-a}\int_{a}^{b}f\left( x\right) g\left( x\right) dx\leq \alpha
\times L\left( \left[ f\left( a\right) \right] ^{\frac{1}{\alpha }},\left[
f\left( b\right) \right] ^{\frac{1}{\alpha }}\right) +\beta \times L\left( %
\left[ g\left( a\right) \right] ^{\frac{1}{\beta }},\left[ g\left( b\right) %
\right] ^{\frac{1}{\beta }}\right)
\end{equation*}%
where $L(p,q):=\frac{p\text{-}q}{\ln p-\ln q}$ $(p\neq q)$ is the
logarithmic mean of the positive real numbers.
\end{corollary}

\begin{proof}
\bigskip We take $s=1$ in (\ref{oo}), we get the required result.
\end{proof}

\begin{theorem}
\label{t40}\bigskip \bigskip Let $f,g,a,b$\ be as in Theorem \ref{f} and $%
p,q>1$ with $\frac{1}{p}+\frac{1}{q}=1.$Then the following inequality holds: 
\begin{equation}
\frac{1}{b-a}\int_{a}^{b}f\left( x\right) g\left( x\right) dx\leq K\left(
s,p,q;k\left( \omega \right) ,k\left( \ell \right) \right)   \label{40}
\end{equation}%
where $\omega \left( u,v\right) $, $\ell \left( u,v\right) ,$ $k\left(
\omega \right) ,$ $k\left( \ell \right) $ is defined as above and%
\begin{eqnarray*}
&&K\left( s,p,q;k\left( \omega \right) ,k\left( \ell \right) \right)  \\
&=&\left[ f\left( b\right) g\left( b\right) \right] ^{s}\left( k\left(
\omega \left( sp,sp\right) \right) \right) ^{\frac{1}{p}}\left( k\left( \ell
\left( sq,sq\right) \right) \right) ^{\frac{1}{q}},\text{ }f\left( a\right)
,g\left( a\right) ,f\left( b\right) ,g\left( b\right) \leq 1.
\end{eqnarray*}%
\ \ \ 
\end{theorem}

\begin{proof}
Since $f$ and $g$ are positive funtions and using the well known H\"{o}lder
inequality on the right side of (\ref{o1}), we have 
\begin{eqnarray}
&&\frac{1}{b-a}\int_{a}^{b}f\left( x\right) g\left( x\right) dx  \label{41}
\\
&\leq &\left( \int_{0}^{1}\left[ f\left( ta+\left( 1-t\right) b\right) %
\right] ^{p}dt\right) ^{\frac{1}{p}}\left( \int_{0}^{1}\left[ g\left(
ta+\left( 1-t\right) b\right) \right] ^{q}dt\right) ^{\frac{1}{q}}  \notag \\
&\leq &\left( \int_{0}^{1}\left[ \left[ f\left( a\right) \right] ^{pt^{s}}%
\left[ f\left( b\right) \right] ^{p\left( 1-t\right) ^{s}}\right] dt\right)
^{\frac{1}{p}}\left( \int_{0}^{1}\left[ \left[ g\left( a\right) \right]
^{qt^{s}}\left[ g\left( b\right) \right] ^{q\left( 1-t\right) ^{s}}\right]
dt\right) ^{\frac{1}{q}}  \notag
\end{eqnarray}%
When $0<f\left( a\right) ,$ $g\left( a\right) ,$ $f\left( b\right) ,$ $%
g\left( b\right) \leq 1,$ by (\ref{b}), we get that%
\begin{eqnarray}
\int_{0}^{1}\left[ \left[ f\left( a\right) \right] ^{pt^{s}}\left[ f\left(
b\right) \right] ^{p\left( 1-t\right) ^{s}}\right] dt &\leq &\int_{0}^{1}%
\left[ f\left( a\right) \right] ^{spt}\left[ f\left( b\right) \right]
^{sp\left( 1-t\right) }dt  \notag \\
&=&\left[ f\left( b\right) \right] ^{sp}k\left( \omega \left( sp,sp\right)
\right) .  \notag \\
\int_{0}^{1}\left[ \left[ g\left( a\right) \right] ^{qt^{s}}\left[ g\left(
b\right) \right] ^{q\left( 1-t\right) ^{s}}\right] dt &\leq &\int_{0}^{1}%
\left[ g\left( a\right) \right] ^{sqt}\left[ g\left( b\right) \right]
^{sq\left( 1-t\right) }dt  \notag \\
&=&\left[ g\left( b\right) \right] ^{sq}k\left( \ell \left( sq,sq\right)
\right) .  \label{42}
\end{eqnarray}%
Therefore,%
\begin{eqnarray*}
&&\frac{1}{b-a}\int_{a}^{b}f\left( x\right) g\left( x\right) dx \\
&\leq &\left[ f\left( b\right) g\left( b\right) \right] ^{s}\left( k\left(
\omega \left( sp,sp\right) \right) \right) ^{\frac{1}{p}}\left( k\left( \ell
\left( sq,sq\right) \right) \right) ^{\frac{1}{q}}.
\end{eqnarray*}%
Thus, Theorem \ref{t40} is proved.
\end{proof}

\begin{corollary}
\bigskip Let $f,g:I\subset 
%TCIMACRO{\U{211d} }%
%BeginExpansion
\mathbb{R}
%EndExpansion
_{0}\rightarrow 
%TCIMACRO{\U{211d} }%
%BeginExpansion
\mathbb{R}
%EndExpansion
_{+}$\ be a log-convex mappings on $I$.\ If $a,b\in I$ with $a<b,$ then the
following inequality holds:%
\begin{equation*}
\frac{1}{b-a}\int_{a}^{b}f\left( x\right) g\left( x\right) dx\leq \left(
L\left( \left[ f\left( a\right) \right] ^{p},\left[ f\left( b\right) \right]
^{p}\right) \right) ^{\frac{1}{p}}\left( L\left( \left[ g\left( a\right) %
\right] ^{q},\left[ g\left( b\right) \right] ^{q}\right) \right) ^{\frac{1}{q%
}}
\end{equation*}%
where $L(p,q):=\frac{p\text{-}q}{\ln p-\ln q}$ $(p\neq q)$ is the
logarithmic mean of the positive real numbers.
\end{corollary}

\begin{proof}
\bigskip We take $s=1$ in (\ref{40}), we get the required result.
\end{proof}

\bigskip

\bigskip

\bigskip

\end{document}